\documentclass[12pt]{article}
\usepackage{amsfonts,amsmath}
\usepackage{hyperref,amsthm}
\usepackage{fancyhdr}
\usepackage{setspace}
\usepackage{amsmath}
\usepackage{enumerate}
\usepackage{amsxtra}
\usepackage{amscd}
\usepackage{amsthm}
\usepackage{amsfonts}
\usepackage{amssymb}
\usepackage{mathrsfs}
\usepackage[all]{xy}
\usepackage[pdftex]{color,graphicx}
\usepackage{graphicx}
\usepackage{color}
\usepackage{mathrsfs}
\newcommand{\stickT}{%
\setbox255=\hbox{\raise1ex\hbox{$\hspace{0.2pt}\,\star\,$}}
\mathord{\rlap{\hbox to\wd255{\hss\hbox{$|$}\hss}}
\box255}
}
\newcommand{\stickS}{%
\setbox255=\hbox{\raise0.6ex\hbox{$\scriptstyle\star$}}
\mathord{\rlap{\hbox to\wd255{\hss\hbox{$\scriptstyle|$}\hss}}
\box255}
}

\newcommand{\act}{\curvearrowright}

\newtheorem{thm}{Theorem}[section]
\newtheorem{defin}[thm]{Definition}
\newtheorem{prop}[thm]{Proposition}
\newtheorem{lemma}[thm]{Lemma}
\newtheorem{claim}[thm]{Claim}
\newtheorem{corol}[thm]{Corollary}
\newtheorem{que}[thm]{Question}

\newtheorem{rem}[thm]{Remark}

\renewenvironment{proof}[1][Proof]{ \noindent \textbf{#1: }}{$\Box$
\bigskip}

\def\Z{\Bbb Z}

\def\Q{\Bbb Q}
\def\P{\Bbb P}

\def\R{\Bbb R}

\def\P{\Bbb P}

\def\e{\varepsilon}

\def\d{\delta}

\def\U{\mathcal{U}}

\def\p^k{\tilde{\P}^k}

\DeclareMathOperator\id{id}

 \DeclareMathOperator\PSL{PSL}
 
 \DeclareMathOperator\SL{SL}

\DeclareMathOperator\rank{rank}
\DeclareMathOperator\dist{dist}

\makeatletter
\newcommand{\RN}[1]{%
  \textup{\lowercase\expandafter{\romannumeral#1}}%
}\makeatother
\begin{document}

\title{
Maximal subgroups of $\SL(n,\Z)$\\
}

\author{
Tsachik Gelander and Chen Meiri \\
%{\tt chenm\startx.technion.ac.il}\\
%}
}

\maketitle
 
\newcounter{enumTemp}

{\begin{center}{\it Dedicated to the 75'th birthday of Herbert Abels}
\end{center}}

\section{Introduction}

More than three decades ago, in a remarkable paper \cite{MS81}, Margulis and Soifer proved existence of maximal subgroups of infinite index in $\SL(n,\Z)$, answering a question of Platonov. Moreover, they proved that there are uncountably many such subgroups. Since then, it is expected that there should be examples of various different nature. However, as the proof is non-constructive and relies on the axiom of choice, it is highly non-trivial to lay one's
hands on specific properties of the resulting groups. 

Our purpose here is to show that indeed, maximal subgroups $\Delta\le\SL(n,\Z)$ of different nature do exist. However, our methods say nothing about the intrinsic algebraic structure of $\Delta$. We do not gain any understanding about the abstract groups $\Delta$, instead we are focusing on the way it sits inside $\SL(n,\Z)$. The two point of views that we consider are:
\begin{itemize}
\item The associated permutation representation $\Gamma\act\Gamma/\Delta$.
\item The action of $\Delta$ on the associated projective space $\P=\P^{n-1}(\R)$. 
\end{itemize}

The main results of this paper are: 

\begin{thm}\label{thm counting} 
Let $n \ge 3$. There are $2^{\aleph_0}$ infinite index maximal subgroups in $\SL(n,\Z)$.
\end{thm}

\begin{thm}\label{thm dense} 
Let $n \ge 3$. There exists a maximal subgroup $\Delta$ of $\SL(n,\Z)$ which does not have a dense orbit in $\P$.  In particular, the limit set of $\Delta$ (in the sense of $\cite{CG00}$) is nowhere-dense. 
\end{thm}

\begin{thm}\label{thm trivial} 
Let $n \ge 3$. There exists an infinite index maximal subgroup $M$ of  $\PSL(n,\Z)$  and an element $g \in \PSL(n,\Z)$
such that $M \cap gMg^{-1}=\{id\}$.
\end{thm}  

\begin{thm}\label{thm not 2-trans} 
Let $n \ge 3$. There exists a primitive permutation action of $\SL(n,\Z)$ which is not 
2-transitive. 
\end{thm} 

\begin{rem}\normalfont The theorems remain true also for 
$\SL(n,\Q)$ instead of $\SL(n,\Z)$.
\end{rem}

\begin{rem}\normalfont
Recall that the Margulis--Soifer theorem is much more general, i.e. holds for any finitely generated non-virtually-solvable linear group $\Gamma$. Our results however relies on special properties of $\SL(n,\Z),~n\ge 3$. In particular
one important ingredient for us is the beautiful result of Venkataramana about commuting unipotents, Theorem \ref{Thm - Venkataramana}. Another ingredient is the result of Conze and Guivarc'h, Theorem \ref{them - conze}. Some of our results can be extended to the class of arithmetic groups of higher $\Q$-rank.
\end{rem}

{\flushleft{\bf Acknowledgment:}} The first author was partially supported by ISF-Moked grant 2095/15. The second author was partially supported by ISF grant 662/15.

\section{Preliminaries}

In this paper we always assume that $n \ge 3$.

\subsection{Projective space}
Let $n \ge 3$. The letter $\P$ denotes the $(n-1)$-dimensional real projective space and fix some compatible metric $\dist_\P$ on $\P$. For every $0 \le k \le n-1$, the set $\mathbb{L}_k$ of $k$-dimensional subspaces of $\P$ can be endowed with the metric defined by 
$$
 \dist_{\mathbb{L}_k}(L_1,L_2):=\max\{\dist_\P(x,L_i) \mid x \in L_{3-i} \text{ for } 1 \le i \le 2\}
 $$ 
 for every $L_1,L_2 \in \mathbb{L}_k$. Note that 
$\mathbb{L}_k$ is naturally homeomorphic to the Grassmannian $\mathrm{Gr}(k+1,\R^n)$.  For $\e>0$ and a subset $A \subseteq \P$
we denote $(A)_\e:=\{x \in \P \mid \dist_\P(x,A) < \e\}$ and  $[A]_\e:=\{x \in \P \mid \dist_\P(x,A) \le \e\}$. 
%Note that $[A]_\e$ is the closure of $(A)_\e$. 
If $A=\{p\}$ then we usually write $(p)_\e$ and  $[p]_\e$ instead of $(A)_\e$ and  $[A]_\e$.

\subsection{Unipotent elements}
\begin{defin}[Rank 1 unipotent elements]\label{def - rank 1 unipotent} We say that a unipotent element $u$ has rank 1 if $\rank(u-\textrm{I}_n)=1$. The point $p_u\in \P$ which is induced by the euclidean line ${\{ux-x\mid x \in \R^n\}} $ is called the point of attraction of $u$.  The $(n-2)$-dimensional subspace  $L_u\subseteq \P$ which is induced by the euclidean $(n-1)$-dimensional space ${\{x \in \R^n \mid ux=x\}}$ is called the fixed hyperplane of of $u$. The set of rank-1 unipotent elements in $SL(n,\Z)$ is denoted by $\mathcal{U}$. 
\end{defin}

The following two lemmas follow directly from the definition of $\U$ and are stated for future reference. 

\begin{lemma}[Structure of unipotent elements]\label{lemma - structure of unipotent} 
The set $\U$ can be divided into equivalence classes in the following way: $u,v \in \U$ are equivalent if 
there exist non-zero integers $r$ and $s$ such that $u^s=v^r$.
The map $u \mapsto (p_u,L_u)$ is a bijection between equivalence classes in $\U$ and the set of pairs $(p,L)$ where $p \in \P$ is a rational point and $ L \subseteq \mathbb{L}_{n-2}$ 
is an $(n-2)$-dimensional rational subspace which contains $p$.
\end{lemma}

\begin{lemma}[Dynamics of unipotent elements]\label{lemma - dynamic of unipotent} Let $u \in \mathcal{U}$.  For every $\varepsilon>0$ and every $\delta>0$ there exists a constant $c$ such that if 
$m \ge c$
and $v=u^m $, then  $v^{k}(x) \in (p_u)_{\varepsilon}$ for every $x \in \P\setminus (L_u)_{\delta}$ and every  $k \ne 0$. Note that the previous lemma implies that $p_u=p_v$ and $L_u=L_v$.
\end{lemma}

\subsection{Schottky systems}

\begin{defin} Assume that $\mathcal{S}$ is a non-empty subset of $\mathcal{U}$ and $\mathcal{A} \subseteq \mathcal{R}$ are closed subsets of $\P$. We say that $\mathcal{S}$ is a Schottky set with respect to the attracting set  $\mathcal{A}$ and the repelling set $\mathcal{R}$ and call the triple 
$(\mathcal{S},\mathcal{A},\mathcal{R})$ a Schottky system if for every $u \in \mathcal{S}$ there exist two positive
numbers  $\delta_u \ge \e_u$ such that the following properties hold:
\begin{enumerate}
\item $u^{k}(x) \in (p_u)_{\varepsilon_u}$ for every $x \in \P\setminus (L_u)_{\delta_u}$ and every  $k \ne 0$;
\item If $u \ne v \in \mathcal{S}$ then $(p_u)_{\varepsilon_u} \cap (L_v)_{\delta_v}=\emptyset$;
\item $\cup_{u \in \mathcal{S}}(p_u)_{\varepsilon_u} \subseteq \mathcal{A}$;
\item $\cup_{u \in \mathcal{S}}(L_u)_{\delta_u} \subseteq \mathcal{R}$.
\end{enumerate}
\end{defin} 
\begin{defin} 
The Schottky system $(\mathcal{S},\mathcal{A},\mathcal{R})$ is said to be profinitely-dense if $\mathcal{S}$ generates a profinitely-dense subgroup of $\SL(n,\Z)$.
We say that the Schottky system $(\mathcal{S}_+,\mathcal{A}_+,\mathcal{R}_+)$ contains the Schottky system $(\mathcal{S},\mathcal{A},\mathcal{R})$ if 
$\mathcal{S}_+ \supseteq \mathcal{S}$, $\mathcal{A}_+ \supseteq \mathcal{A}$ and $\mathcal{R}_+ \supseteq \mathcal{R}$. 
\end{defin}

\begin{lemma}\label{adding} Let $(\mathcal{S},\mathcal{A},\mathcal{R})$ be a Schottky system. 
Assume that $[p]_\e \cap \mathcal{A}=\emptyset$ and $[L]_\delta \cap \mathcal{R}=\emptyset$ where $\d \ge\e >0$ and 
$p$ is a rational point which is continued in a rational 
subspace $L \in \mathbb{L}_{n-2}$.  Denote $\mathcal{A}_+=\mathcal{A} \cup [p]_\e$ and $\mathcal{R}_+=\mathcal{R} \cup [L]_\delta$. 
Then there exist $v \in \mathcal{U}$ with $p=p_v$, $L=L_v$ such that  $(\mathcal{S}_+,\mathcal{A}_+,\mathcal{R}_+)$ is  Schottky system which contains $(\mathcal{S},\mathcal{A},\mathcal{R})$ where $\mathcal{S}_+:=\mathcal{S}\cup\{v\}$.
\end{lemma}

\begin{proof} Lemma \ref{lemma - structure of unipotent} implies that there exists  $u \in \mathcal{U}$ such that $p_u=p$ and $L_u=L$.
Lemma \ref{lemma - dynamic of unipotent}  implies that there exists $m \ge 1$ such that $v:=u^m$ satisfies the required properties.  
\end{proof}

The following lemma is a version of the well known ping-pong lemma:

\begin{lemma}[Ping-pong]\label{lemma - ping pong} Let $(\mathcal{S},\mathcal{A},\mathcal{R})$ be a 
Schottky system. Then the natural homomorphism $*_{u \in \mathcal{S}}\langle u\rangle \rightarrow \langle \mathcal{S}\rangle $ is an isomorphism. 
\end{lemma}

An important ingredient for our methods is the following beautiful result:

\begin{thm}[Venkataramana, \cite{Ve87}]\label{Thm - Venkataramana} Let $\Gamma$ be a Zariski-dense subgroup of $\SL(n,\Z)$. Assume that
$u\in \mathcal{U} \cap \Gamma$, $v \in \Gamma$ is unipotent  and $\langle u,v\rangle \simeq \Z^2$. Then $\Gamma$ has finite index in $\SL(n,\Z)$. In particular, if  
$\Gamma$ is profinitely-dense then $\Gamma=\SL(n,\Z)$. 
\end{thm}

Note that if $u,v \in \SL(n,\Z)\cap \mathcal{U}$ and $p_u=p_v$ then $(u-1)(v-1)=(v-1)(u-1)=0$ and in particular
$uv=vu$. 
Thus we get the following lemma:
\begin{lemma}\label{lemma - finding a good pair} Let $g \in \SL(n,\Z)$ and $u_1,u_2 \in \mathcal{U}$.
Assume that $p_{u_2}=gp_{u_1}$ and $L_{u_2} \ne gL_{u_1}$. Then $\langle u_1,g^{-1}u_2g \rangle \simeq \Z^2$.
\end{lemma}

\begin{lemma}\label{throwing} Assume that $g$ is an element of $\SL(n,\Z)$, $(\mathcal{S},\mathcal{A},\mathcal{R})$ is a profinitely-dense Schottky system, $\delta \ge \e>0$, 
$p_1$ and $p_2$ are rational  points and $L_1$ and $L_2$ are rational  $(n-2)$-dimensional subspaces such that the following conditions hold:
\begin{enumerate}
\item $([p_1]_\e  \cup [p_2]_\e)\cap \mathcal{R}=\emptyset$ and $([L_1]_\delta \cup [L_2]_\delta)\cap \mathcal{A}=\emptyset$;
\item $[p_1]_\e \cap [L_2]_\delta=\emptyset$ and $[p_2]_\e \cap [L_1]_\delta=\emptyset$;
\item $p_1=gp_2$ and $L_1 \ne gL_2$.
\end{enumerate}
Denote $\mathcal{A}_+=\mathcal{A} \cup [p_1]_\e \cup [p_2]_\e$ and $\mathcal{R}_+=\mathcal{R} \cup [L_1]_\delta \cup [L_2]_\delta$.
Then there exists a set $\mathcal{S}_+ \supseteq \mathcal{S}$ such that $(\mathcal{S}_+,\mathcal{A}_+,\mathcal{R}_+)$ is  a Schottky system  which contains $(\mathcal{S},\mathcal{A},\mathcal{R})$
and $\langle \mathcal{S}_+,g \rangle=\SL(n,\Z)$. 
\end{lemma}

\begin{proof} For every $1 \le i \le 2$ choose $u_i\in \mathcal{U}$ such that
$p_{u_i}=p_i$ and $L_{u_i}=L_i$.
  Lemma \ref{lemma - finding a good pair} implies that $\langle u_1,g^{-1}u_2g \rangle \simeq \Z^2$.
Lemma \ref{lemma - dynamic of unipotent}  implies that there exists $m \ge 1$ such that $(\mathcal{S}_+,\mathcal{A}_+,\mathcal{R}_+)$ is  Schottky system
where $v_1:=u_1^m$, $v_2:=u_2^m$ and $\mathcal{S}_+:=\mathcal{S} \cup \{v_1,v_2\}$. Theorem \ref{Thm - Venkataramana} implies  
that $\langle \mathcal{S}_+,g \rangle=\SL(n,\Z)$. 
\end{proof}

\begin{defin} Let $1 \le k \le n$. A $k$-tuple $(p_1,\ldots,p_k) $ of projective points is called generic if  $p_1,\ldots,p_k$ span a $(k-1)$-dimensional
subspace of $\P$. Note that the set of generic $k$-tuples of $\P$ is an open subset of the product of $k$ copies of the projective space, indeed it is even Zariski open.
\end{defin}

\begin{thm}[Conze-Guivarc'h, \cite{CG00}]\label{them - conze} Assume that $n \ge 3$ and that $\Gamma \le \SL(n,\R)$ is a lattice. Then $\Gamma$ acts minimally
of the set of generic $(n-1)$-tuples.  
\end{thm}

\begin{corol}\label{cor - conze} Assume that  $n \ge 3$ and $\Gamma \le \SL(n,\R)$  is a lattice. For  every $1 \le i \le 2$ let $p_i \in L_i \in \mathbb{L}_{n-2}$. Then for every positive numbers $\varepsilon$
and $\delta$ there exists $g \in \Gamma$ such that $gp_1 \in (p_2)_{\varepsilon}$ and
$gL_1 \in (L_2)_{\delta}$.
\end{corol}

The proof of the following Proposition is based on the proof of the main result of \cite{AGS14}. 

\begin{prop}\label{prop -pro-dense} Assume that $n \ge 3$ and $p \in L \in \mathbb{L}_{n-2}$.
Then for every $\delta \ge \e >0$ there exists a finite subset $\mathcal{S}\subseteq \mathcal{U}$ such that  $(\mathcal{S},\mathcal{A},\mathcal{R})$ is a
profinitely-dense Schottky system  where $\mathcal{A}:=[p]_\e$ and  $\mathcal{R}:=[L]_\delta$.
\end{prop}
\begin{proof} We recall some facts about Zariski-dense and profinitely-dense subgroups.  For a positive integer $d \ge 2$ let 
$\pi_d:\SL(n,\Z)\rightarrow \SL(n,\Z/d\Z)$ be the modulo-d homomorphism and denote 
$K_d:=\ker \pi_d$. 

\begin{itemize}
\item[(a)] If $H \le \SL(n,\Z)$ and $\pi_p(H)=\SL(n,\Z/p\Z)$ for some odd prime $p$ then 
$H$ is Zariski-dense, \cite{We96} and \cite{Lu99}.  
\item[(b)] The strong approximation theorem  of Weisfeiler \cite{We84} and Nori \cite{No87} implies that if a subgroup $H$ of  $\SL(n,\Z)$ is Zariski-dense then there exists some positive integer $q$ such that $\pi_d(H)=\SL(n,\Z/d\Z)$ whenever $\gcd(q,d)=1$. 
\item[(c)] If $H \le \SL(n,\Z)$, $\pi_4(H)=\SL(n,\Z/4\Z)$ and $\pi_p(H)=\SL(n,\Z/p\Z)$ for all
 odd primes $p$ then $H$ is profintiely-dense in $\SL(n,\Z)$. 
 \end{itemize}

Fix $\delta \ge \e>0$ and set $\mathcal{A}:=[p]_\e$ and $\mathcal{R}:=[L]_\delta$. For every $1 \le i \le 2n^2-n$, fix a point $p_i$ beloning to an $(n-2)$-dimensional
subspace $L_i$ and positive numbers  $\delta_i \ge \e_i>0$ such that the following two conditions hold:
\begin{enumerate}
\setcounter{enumi}{\theenumTemp}
\item $\cup_{1 \le i \le 2n^2-n}(p_i)_{\varepsilon_i}\subseteq \mathcal{A}$
 and  $\cup_{1 \le i \le 2n^2-n}(L_i)_{\delta_i}\subseteq \mathcal{R}$;
 \item For every $1 \le i \ne j \le 2n^2-n$, $(p_i)_{\varepsilon_i} \cap (L_j)_{\delta_j}=\emptyset$.
\setcounter{enumTemp}{\theenumi}
\end{enumerate}  
For every $1 \le i \ne j \le n$, let $e_{i,j}\in \SL(n,\Z)$ be the matrix with 1 on the diagonal and on the $(i,j)$-entry and zero elsewhere and
let $e_1,\ldots,e_{n^2-n}$ be an enumeration of the $e_{i,j}$'s. Denote the exponent of $\SL(n,\Z/3\Z)$ by $t$.
If $g_1,\ldots,g_{n^2-n}\in K_3$ and  $k_1,\ldots,k_{n^2-n}$ are positive integers then
$\pi_3(H_1)=\SL(n,\Z/3\Z)$ where $u_i:=g_ie_i^{tk_i+1}g_i^{-1}$ and $H_1:=\langle u_i \mid 1 \le i \le n^2-n \rangle$.  Note that for 
every $u \in \mathcal{U}$ and $g \in \SL(n,\Z)$, $p_{gug^{-1}}=gp_u$ and $L_{gug^{-1}}=gL_u$. Thus, Lemma \ref{lemma - dynamic of unipotent} and Corollary \ref{cor - conze} imply that it is possible to choose $g_i$'s and $k_i$'s
such that:
\begin{enumerate}
\setcounter{enumi}{\theenumTemp}
\item  $u_i^k(x) \in (p_i)_{\varepsilon_i}$ for every $1 \le i \le n^2-n$, every $x \not \in (L_i)_{\delta_i}$ and every $k \ne 0$.
\setcounter{enumTemp}{\theenumi}
\end{enumerate}
In particular, $\{ u_1,\ldots, u_{n^2-n}\}$ is a Schottky set with respect to $\mathcal{A}$ and $\mathcal{R}$ which generates a Zariski-dense subgroup $H_1$. 

The strong approximation theorem implies that  there exists some positive integer $q$ such that $\pi_d(H_1)=\SL(n,\Z/d\Z)$ whenever $\gcd(q,d)=1$.
Denote the exponent of $\SL(n,\Z/q^2\Z)$ by $r$.
 As before, there exist $g_{n^2-n+1},\ldots,g_{2n^2-2n}\in K_{q^2}$ and positive integers $k_{n^2-n+1},\ldots,k_{2n^2-2n}$ such that the elements of the form $u_i:=g_ie_i^{rk_i+1}g_i^{-1}$ satisfy:
\begin{enumerate}
\setcounter{enumi}{\theenumTemp}
\item $\pi_{q^2}(H_2)=\SL(n,\Z/q^2\Z)$ where   $H_2:=\langle u_i \mid n^2-n+1 \le i \le 2n^2-2n \rangle$;
\item   $u_i^k(x) \in (p_i)_{\varepsilon_i}$ for every $n^2-n+1 \le i \le 2n^2-2n$, every $x \not \in (L_i)_{\delta_i}$ and every $k \ne 0$. 
\end{enumerate}

Denote $\mathcal{S}:=\{ u_1,\ldots, u_{2n^2-2n}\}$.  Item (c) implies that $\pi_d(\langle \mathcal{S} \rangle)=\SL(n,\Z/d\Z)$ for every $d \ge 1$. 
Thus, $(\mathcal{S},\mathcal{A},\mathcal{R})$ is the required profinitely-dense Schottky system.
\end{proof}

The following lemma will be needed in Section \ref{proof 3}.
\begin{lemma}\label{starting} Assume that $k$ is an element of $\SL(n,\Z)$, $\delta \ge \e>0$,
$p_1$, $p_2$  and $p_3$ are rational  points and $L_1$, $L_2$ and $L_3$ are rational  $(n-2)$-dimensional subspaces such that the following conditions hold
\begin{itemize}
\item For every $1 \le i \ne j \le 3$, $p_i \in L_i$ and $[p_i]_\e \cap [L_j]_\delta=\emptyset$;
\item $p_2=kp_1$, $p_3:=k^2p_1$, $L_2=kL_1$, $L_3 \ne k^2 L_1$;
\end{itemize}
Denote $\mathcal{A}:=\cup_{1 \le i \le 3}[p_i]_{\e}$ and $\mathcal{R}:=\cup_{1 \le i \le 3}[L_i]_{\delta}$. Then 
there exists a profinitely-dense
Schottky system $(\mathcal{S},\mathcal{A},\mathcal{R})$ such that $\langle \mathcal{S} \rangle \cap k\langle \mathcal{S} \rangle k^{-1} \ne \{\id\}$
and $\langle \mathcal{S},k \rangle=\SL(n,\Z)$.
\end{lemma}

\begin{proof} Proposition \ref{prop -pro-dense} implies that there exists a finite  $\mathcal{S}_0=\{u_1,\ldots,u_r\} \subseteq \mathcal{U}$ such that 
$(\mathcal{S}_0,\mathcal{A}_0,\mathcal{R}_0)$ is profinitely-dense Schottky system where $\mathcal{A}_0:=[p_1]_\e$ and 
$\mathcal{R}_0:=[L_1]_\delta$. A quick look at the proof of this proposition implies that we can assume that $p_{u_1}=p_1$ and $L_{u_1}=L_1$. 
Denote $w_2:=ku_1k^{-1}$ and choose $w_3 \in \mathcal{U}$ such that $p_{w_3}=p_3$, $L_{w_3}=L_3$. Lemma \ref{lemma - finding a good pair} implies that $\langle u_1, k^{-2}w_3^mk^2 \rangle \cong \Z^2$ for every $m \ne 0$. 
Lemma \ref{lemma - dynamic of unipotent} implies that for large enough value of $m$, $(\mathcal{S},\mathcal{A},\mathcal{R})$ is a
profinitely-dense Schottky system where $\mathcal{S}:=\mathcal{S}_0 \cup \{w_2^m,w_3^m\}$. Note that 
$w_2 \in \langle \mathcal{S} \rangle \cap k\langle \mathcal{S} \rangle k^{-1}$ and that Theorem \ref{Thm - Venkataramana} implies that 
$\langle \mathcal{S},k \rangle=\SL(n,\Z)$.
\end{proof}

\section{Proof of Theorem \ref{thm counting}}

Zorn's lemma implies that every proper subgroup $H$ of $\SL(n,\Z)$ is contained in a maximal subgroup $M$
(Since $\SL(n,\Z)$ is finitely generated an increasing union of proper subgroups is a proper subgroup). 
If $H$ is profinitely-dense then so is $M$; hence $M$ should
have infinite index. Thus, Theorem \ref{thm counting} follows for the following proposition:

\begin{thm} 
Let $n \ge 3$. There exist $2^{\aleph_0}$ infinite-index profinitely-dense subgroups of $\SL(n,\Z)$
such that the union of any two of them generates $\SL(n,\Z)$.
\end{thm}
    
\begin{proof} For every non-negative integer $i$ fix a rational point $p_i$ belonging to a rational $(n-2)$-dimensional 
subspace $L_i$ and two numbers $\delta_i \ge \e_i >0$ such that $[p_i]_{\varepsilon_i} \cap [L_j]_{\delta_j}=\emptyset$  for every $i \ne j$.
Let $\mathcal{A}$ and $\mathcal{R}$ be the closures of $\cup_{i \ge 0}(p_i)_{\epsilon_i}$ and $\cup_{i \ge 0}(L_i)_{\delta_i}$ respectively. Proposition \ref{prop -pro-dense} implies that there exists a finite subset $\mathcal{S}_0\subseteq \mathcal{U}$ such that $(\mathcal{S}_0,\mathcal{A}_0,\mathcal{R}_0)$ is a profinitely-dense
Schottky system where  $\mathcal{A}_0 = [p_0]_{\varepsilon_{0}}$ and $\mathcal{R}_0 = [L_0]_{\delta_0}$.  Lemmas \ref{lemma - structure of unipotent} and \ref{lemma - dynamic of unipotent} imply that for every $i \ge 1$ there are $u_{i,1},u_{i,2} \in \mathcal{U}$ such that:
\begin{enumerate}
\item $p_i=p_{u_{i,1}}=p_{u_{i,2}}$ and $L_{u_{i,1}} \ne L_{u_{i,2}} \subseteq (L_i)_{\delta_i}$ (hence, $\langle u_{i,1},u_{i,2}\rangle\cong \Z^2$);
\item $u_{i,j}^k(x) \in (p_i)_\varepsilon$ for every $1 \le j \le 2$, every $x \not \in (L_i)_{\delta_i}$ and every $k \ne 0$. 
\end{enumerate}
For every function $f$ from the positive integers to $\{0,1\}$ the set $\mathcal{S}_f:=\mathcal{S}_0\cup \{u_{i,f(i)\mid i\ge 1}\}$ is a Schottky set  with respect to the attracting set  $\mathcal{A}$ and the repelling set $\mathcal{R}$. If $f$ and $g$ are distinct function then $\mathcal{S}_f \cup \mathcal{S}_g $ contains $\{u_{i,1},u_{i,2}\}$ for some $i \ge 1$ so Theorem \ref{Thm - Venkataramana} implies that 
$\langle \mathcal{S}_f \cup \mathcal{S}_g \rangle =\SL(n,\Z)$. 
\end{proof}

%\begin{remark}  
Theorem \ref{thm counting} implies that there are non-conjugate infinite-index maximal subgroups in $\SL(n,\Z)$.  
Indeed, since $\SL(n,\Z)$ is countable so is the conjugacy class of every infinite index subgroup.
We can say a little more: 

\begin{prop}
There are infinite-index maximal subgroups whose actions on $\P$ have
different topological properties. For example, when $n \ge 4$ some but not all
infinite-index maximal subgroups $M$ of $\SL(n,\Z)$ have the following property:

There exists a finitely-generated subgroup $H \le M$ and a $2$-dimensional subspace $l \in \mathbb{L}_2$ which is contained in
the closure of every $H$-orbit. 
\end{prop}

The existence of maximal subgroups with this property follows form the existence
of maximal subgroup $M$ which contains an element $g \in \SL(n,\Z)$ 
with eigenvalues $\alpha_1,\ldots,\alpha_n$ such that $\alpha_1^m$ and $\alpha_n^m$ are not real for all $m \ge 1$
and $|\alpha_1|<|\alpha_i|<|\alpha_n|$ for all $2<i<n-1$. 
The construction of maximal subgroups without this property can be done by using similar ideas to the ones used
in the proofs of Theorems \ref{thm dense}, \ref{thm trivial} and \ref{thm not 2-trans} below. 
  
%\end{remark}

%%%%%%%%%%%%%%%%%%%%%%%%
\section{Proof of Theorem \ref{thm dense}}
%%%%%%%%%%%%%%%%%%%%%%%%

In order to prove the Theorem \ref{thm dense} it is enough to find an infinite-index maximal subgroup $\Delta$ and two open subsets $U$ and $V$ such that $gU \cap V =\emptyset$ for all $g \in \Delta$. Let $V^\vee$ be the dual of $V:=\R^n$. The natural map form $\P=\mathrm{Gr}(1,V)$ to $\mathrm{Gr}(n-1,V^\vee)$ is an $\SL(n,\Z)$-equivariant  
homeomorphism. Thus, Theorem \ref{thm dense} follows from the following proposition:

\begin{prop}\label{prop dense} Let $L_1$ and $L_2$  be distinct $(n-2)$-dimensional subspaces of $\P$. There exist $\rho>0$ and a free profinitely-dense subgroup 
$H$ such that if $L \subseteq (L_2)_\rho$ and $gL \subseteq (L_1)_\rho$ for some $g \in \SL(n,\Z)$ and some $(n-2)$-dimensional  subspace $L$ then $\langle g,H\rangle =\SL(n,\Z)$.
\end{prop}

The rest of this section is devoted to the proof of Proposition \ref{prop dense}. 
The basic idea is to find $\rho>0$ for which Lemma \ref{throwing} allows us to inductively construct an ascending sequence $(\mathcal{S}_k,\mathcal{A}_K,\mathcal{R}_k)_{k \ge 0}$ of profintely-dense Schottky system such that $\langle \mathcal{S}_k,g_k\rangle=\SL(n,\Z)$
for every $k \ge 1$ where $(g_k)_{k \ge 1}$ is an enumeration of the $g$'s such that 
$g(L_2)_{\rho} \cap (L_1)_\rho \ne \emptyset$.

Before starting with the formal proof let us briefly explain the main idea of the proof.  Let $(\mathcal{S}_0,\mathcal{A}_0,\mathcal{R}_0)$ be a profintly-dense Schottky system  such that $L_1 \cap \mathcal{A}_0 =\emptyset$ and $L_2\setminus \mathcal{R}_0\ne \emptyset$. Assume for simplicity that $g_1L_2 =L_1$. Our goal is to show that we can find a finite Schottky system $(\mathcal{S}_1,\mathcal{A}_1,\mathcal{R}_1)$ which contains $(\mathcal{S}_0,\mathcal{A}_0,\mathcal{R}_0)$
such that $\langle S_1,g_1\rangle=\SL(n,\Z)$. Choose a point $w \in L_2 \setminus \mathcal{R}_0$. We would like to apply Lemma \ref{throwing} with respect to $w$ and $g_1w$  and suitable hyperplanes $w \in L_w$ and $gw_1 \in L_{g_1w}$. The problem is that it is possible that $g_1w \in \mathcal{R}_0$, so the assumptions of Lemma \ref{throwing} do not hold
no matter what $L_w$ and $L_{g_1w}$ are. 
Therefore we first construct a Schottky system $(\mathcal{S}_*,\mathcal{A}_*,\mathcal{R}_*)$ such that $w \not \in \mathcal{R}_*$, $\mathcal{S}_*:=\mathcal{S} \cup \{u\}$ and $g_1w \in L_u$.  
Note that  there is a lot of freedom in choosing the fixed hyperplane $L_u$ and the attracting point $p_u$ of $u$.  The existense of  $(\mathcal{S}_*,\mathcal{A}_*,\mathcal{R}_*)$ is guaranteed by some technical assumptions on $(\mathcal{S}_0,\mathcal{A}_0,\mathcal{R}_0)$; One 
of these assumptions is the requirement that $\mathcal{A}_0 \cap L_1=\emptyset$ which guarantees that $g_1w \not \in \mathcal{A}_0$. 
Denote $h:=g_1^{-1}ug_1$ so $h(w)=w$. Note that if $(\mathcal{S}_1,\mathcal{A}_1,\mathcal{R}_1)$ is a finite Schottky system such that 
$\langle S_1,h\rangle=\SL(n,\Z)$ then also $\langle S_1,g_1\rangle=\SL(n,\Z)$. The advantage now is that we can choose a rational point $w_1$ which is very close to $w$ such that 
$hw_1 \ne w_1$ and $hw_1$ is also very close to $w$. In particular $w_1,hw_1 \not \in \mathcal{R}_*$. Once more we use  some technical assumptions on $(\mathcal{S}_*,\mathcal{A}_*,\mathcal{R}_*)$ to show that there
exist rational hyperplanes $w_1 \in L_{w_1}$ and $hw_1 \in L_{gw_1}$ and positive numbers $\e$ and $\delta$  for which the assumptions of Lemma \ref{throwing} hold. Thus, there exists a finite Schottky system $(\mathcal{S}_1,\mathcal{A}_1,\mathcal{R}_1)$
such that $\langle S_1,h\rangle=\SL(n,\Z)$.  

\medskip

We now start the formal proof. Let $L_0, $$L_1$ and $L_2$  be fixed distinct rational $(n-2)$-dimensional subspaces of $\P$ and let $p_0 \in L_0 \setminus L_1 \cup L_2$ be a fixed point.  We will use $p_0$ and $L_0$ in the last part of the proof to construct the profinitely-dense Schottky system $(\mathcal{S}_0,\mathcal{A}_0,\mathcal{R}_0)$.
\begin{claim}\label{com claim 1}
There exists $\rho>0$
with the following properties:
\begin{enumerate}
\item For every $x \in [L_1]_\rho$ there exists an $(n-2)$-dimensional subspace $L_x$ containing $x$ such that $L_x\cap [p_0]_{\rho}=\emptyset$ and $L_x \setminus [L_0 \cup L_1 \cup L_2]_{\rho} \ne \emptyset$. 
\item For every $(n-2)$-dimensional subspace $L \subseteq [L_2]_{\rho}$ there exist $x \in L \setminus [L_0 \cup L_1]_{\rho}$ and an $(n-2)$-dimensional subspace $ L_x $ containing $x$ such that
$L_x$ is not contained  in $[L_2]_\rho$ and $L_x \cap [p_0]_{\rho}=\emptyset$.
%\item $[p_0]_{\rho} \cap [L_1]_{\rho}=\emptyset$ and $[p_0]_{\rho} \cap [L_2]_{\rho}=\emptyset$
\setcounter{enumTemp}{\theenumi}
\end{enumerate}  
\end{claim}
\begin{proof}[Proof of Claim \ref{com claim 1}]
Fix $\rho_0>0$ such that:
\begin{itemize}
\item $L_0$ and $L_1$ are not contained in $[L_2]_{\rho_0}$;
\item  $p_0 \not \in [L_1]_{\rho_0}$. 
\end{itemize}
For every $x \in [L_1]_{\rho_0}$ there exist an $(n-2)$-dimensional subspace $L_x$ containing $x$ and  $\rho_x>0$ such that $L_x\cap [p_0]_{\rho_x}=\emptyset$ and $L_x \setminus [{L}_0 \cup L_1 \cup L_2]_{\rho_x} \ne \emptyset$. Since
$[p_0]_{\rho_x}$ and $[{L}_0]_{\rho_x}$ are closed sets, if $y$ is close enough to $x$ and the hyperplane $L_y $ containing $y$ is chosen to be close enough to $L_x$ then $L_y\cap [p_0]_{\rho_x}=\emptyset$ and $L_y \setminus [{L}_0 \cup L_1 \cup L_2]_{\rho_x} \ne \emptyset$.
Thus, the compactness of $[L_1]_{\rho_0}$ implies that there exists a uniform $\rho_1>0$ such that for every $x \in [L_1]_{\rho_0}$
there exists an $(n-2)$-dimensional subspace $L_x$ containing $x$  such that $L_x\cap [p_0]_{\rho_1}=\emptyset$ and $L_x \setminus [{L}_0 \cup L_1 \cup L_2]_{\rho_1} \ne \emptyset$.  

The set of $(n-2)$-dimensional subspaces of $\P$ which are contained in $[L_2]_{\rho_0}$ is compact. Thus, a similar argument to the one above implies that there exist $\rho_2>0$ such that for every $(n-2)$-dimensional subspace $L \subseteq [L_2]_{\rho_0}$ there exist $x \in L \setminus [L_0 \cup L_1]_{\rho_2}$ and an $(n-2)$-dimensional subspace $L_x$ containing $x$ such that $L_x$  is not contained in $[L_2]_{\rho_2}$ and $L_x \cap [p_0]_{\rho_2}=\emptyset$. Then $\rho:=min(\rho_0,\rho_1,\rho_2)$ satisfies the requirements. 
\end{proof} 

We fix $\rho>0$ which satisfies the requirements of Claim \ref{com claim 1}.

\begin{claim}\label{com claim 2} Assume that 
$\mathcal{A} \subseteq \mathcal{R}$ are closed sets such that:   
\begin{enumerate}
\item $[L_0\cup L_1]_{\rho}\subseteq \mathcal{R}$ and $[p_0]_{\rho}\subseteq \mathcal{A}$;
\item\label{tmp 125} For every $x \in [L_1]_\rho$ there exists an $(n-2)$-dimensional subspace $L_x$ containing $x$ such that $L_x \cap \mathcal{A}=\emptyset$ and
 $L_x \setminus (\mathcal{R}\cup[L_2]_\rho) \ne \emptyset$;
 \item\label{tmp 123} For every $(n-2)$-dimensional subspace $L \subseteq [L_2]_\rho$ there exist $x \in L \setminus  \mathcal{R}$ and an $(n-2)$-dimensional subspace $L_x $ containing $ x$ such that
 $L_x$ is not contained in  $[L_2]_\rho$ and $L_x \cap \mathcal{A}=\emptyset$.
 %\item $[p_0]_{\e} \cap [L_1]_{\e}=\emptyset$ and $[p_0]_{\e} \cap [L_2]_{\e}=\emptyset$
 \setcounter{enumTemp}{\theenumi}
\end{enumerate}
Then for every $(n-2)$-dimensional subspace $L$ which is disjoint from $\mathcal{A}$, not contained in $[L_2]_\rho$ and contains a point $p \in L \setminus  \mathcal{R}$ there
exists $\delta>0$ such that:
\begin{enumerate}
\setcounter{enumi}{\theenumTemp}
\item\label{tmp 140} Items \ref{tmp 125} and \ref{tmp 123} hold for $\mathcal{A}_{*}:=\mathcal{A} \cup [p]_\d$ and $\mathcal{R}_{*}:=\mathcal{R} \cup [L]_\delta$ instead of $\mathcal{A}$ and $\mathcal{R}$.
\item\label{tmp 141}  $[p]_\d \cap \mathcal{R} = \emptyset$ and $ [L]_\delta \cap \mathcal{A} = \emptyset$.
\setcounter{enumTemp}{\theenumi}
\end{enumerate}
\end{claim}
\begin{proof}[Proof of Claim \ref{com claim 2}] 
For every $x \in [L_1]_\rho$ there exists an $(n-2)$-dimensional subspace $L_x $ containing $ x$ such that item \ref{tmp 125} holds.
By a small deformation of $L_x$ we can assume that $p \not \in L_x$ and  $L_x \ne L$ (so 
$L_x $ in not contained in  $L \cup \mathcal{R}$). Thus, there exists $\d_x>0$ such that $L_x \cap (\mathcal{A}\cup [p]_{\d_x})=\emptyset$ and
 $L_x \setminus (\mathcal{R}\cup[L_2]_\rho \cup [L]_{\d_x}) \ne \emptyset$. If $y$ is close enough to $x$ and the hyperplane $L_y $ containing $y$ is chosen to be close enough to $L_x$ then $L_y\cap (\mathcal{A} \cup [p]_{\d_x})=\emptyset$ and $L_y \setminus (\mathcal{R} \cup[L_2]_\rho \cup [L]_{\d_x} )\ne \emptyset$.
Thus, the compactness of $[L_1]_{\rho}$ implies that there exists a uniform $\d_1>0$ such that for every $x \in [L_1]_{\rho}$
there exists an $(n-2)$-dimensional subspace $L_x$ containing $x$  such that $L_x \cap (\mathcal{A}\cup [p]_{\d_1})=\emptyset$ and
 $L_x \setminus (\mathcal{R}\cup[L_2]_\rho \cup [L]_{\d_1}) \ne \emptyset$ 

Let $L_3$ be an $(n-2)$-dimensional subspaces of $\P$ which is contained in $[L_2]_{\rho}$ so $L_3 \ne L$. 
Item \ref{tmp 123} states that there exists $x \in L_3 \setminus  \mathcal{R}$ and an 
$(n-2)$-dimensional subspace $L_x $ containing $ x$ such that $L_x$ is not contained in $[L_2]_\rho$ and $L_x \cap \mathcal{A} \ne \emptyset$.
The set $\mathcal{R}$ is closed so by a small deformation we can assume that $x \not \in L$ and
$p \not \in L_x$.
Hence, there exists $\delta_{L_3}>0$ such that $x \not \in [L]_{\delta_{L_3}}$ and $L_x \cap [p]_{\delta_{L_3}}=\emptyset$. If $L_4$ is very close to $L_3$ then we can 
pick $y \in L_4 \setminus   \mathcal{R}$ which is very close to $x$ and an $(n-2)$-dimensional subspace $L_y $ containing $y$ which is close to $L_x$ such that $L_y$ is not contained in $[L_2]_\rho$, $L_y \cap \mathcal{A}=\emptyset$,  $y \not \in [L]_{\delta_{L_3}}$ and $L_y \cap [p]_{\delta_{L_3}}=\emptyset$. 
The set of $(n-2)$-dimensional subspaces of $\P$ which are contained in $[L_2]_{\rho}$ is compact and thus there exists a uniform $\delta_2>0$ such that for every $(n-2)$-dimensional subspace $L \subseteq [L_2]_\rho$ there exist $x \in L \setminus \mathcal{R}\cup [L_0]_{\delta_2}$ and an $(n-2)$-dimensional subspace $L_x $ containing $x$ such that $L_x$ is not contained in $[L_2]_\rho$
and $L_x \cap (\mathcal{A} \cup [p]_{\delta_2})=\emptyset$. Finally, if 
 $0 < \delta_3 < \min\{\dist(\mathcal{A},L),\dist(\mathcal{R},p)\}$ 
then $\delta:=min(\delta_1,\delta_2,\delta_3)$ satisfies the requirements. 
\end{proof}

\begin{claim}\label{com claim 3} Assume that $\rho>0$, $\mathcal{A}$ and $\mathcal{R}$ satisfies the assumptions of Claim \ref{com claim 2} and 
$(\mathcal{S},\mathcal{A},\mathcal{R})$ is a profinitely-dense Schottky system.  
Let $g \in \SL(n,\Z)$ be a non-identity element for which there exists an $(n-2)$ dimensional subspace $L $ such that $L \subseteq [L_2]_\rho$ and $gL \subseteq [L_1]_\rho$. Then there exists
$(\mathcal{S}_+,\mathcal{A}_+,\mathcal{R}_+)$ which still satisfies the assumptions of Claim \ref{com claim 2} (with respect to the same $\rho$) and 
$\langle \mathcal{S}_+,g \rangle=\SL(n,\Z)$.
\end{claim}

\begin{proof}[Proof of Claim \ref{com claim 3}] Item \ref{tmp 123} of Claim  \ref{com claim 2} implies  that there exists a  point $w \in L \setminus  \mathcal{A} \cup \mathcal{R}$ 
and an $(n-2)$-dimensional subspace $L_w$ containing $w$ which is not contained in $[L_2]_\rho$ such that $L_w \cap \mathcal{A}=\emptyset$. Denote $w_1:=gw$ and note that $w \ne w_1$ since $w \not \in [L_1]_\rho$ while $w_1 \in [L_1]_\rho$. Item \ref{tmp 125} of Claim  \ref{com claim 2} implies that there exists an $(n-2)$-dimensional subspace $L_{w_1}$ containing $w_1$  and a point $p_1 \in L_{w_1}$ such that   $ L_{w_1} \cap \mathcal{A}=\emptyset$ and $p_1 \in L_{w_1} \setminus (\mathcal{R}\cup [L_2]_\rho)$. 
By a small deformations  we can assume that $w$, $w_1$, $p_1$, $L_w$ and $L_{w_1}$ are rational and that $w \not \in L_{w_1}$ and $p_1 \not \in L_{w}$. Set $r:=\min\{\dist(w,L_{w_1}),\dist(p_1,L_{w}))\}$.
Claim \ref{com claim 2} implies that there exists $0 < \delta_1 < r$
such that items \ref{tmp 140} and \ref{tmp 141} of Claim  \ref{com claim 2} hold for $\mathcal{A_*}:=\mathcal{A}\cup [p_1]_{\delta_1}$ and $\mathcal{R}_*:=
\mathcal{R}\cup [L_{w_1}]_{\delta_1}$ instead of $\mathcal{A}$ and $\mathcal{R}$. Lemma \ref{adding} implies that there exists $u \in \mathcal{U}$ with 
$p_u=p_1$ and $L_u=L_{w_1}$ such that $(\mathcal{S}_*,\mathcal{A}_*,\mathcal{R}_*)$ is Schottky system where $\mathcal{S}_*:=\mathcal{S}\cup\{u\}$. Note that 
$w \not \in \mathcal{A}_* \cup \mathcal{R}_*$ and that $L_w \cap \mathcal{A}_*=\emptyset$. 

Denote $h:=g^{-1}ug$. Since $h(w)=w$ and $h$ is not the identity, there exists a point $p_2$ such that $p_2$ and $p_3 :=hp_3$
are distinct and close as we wish to $w$. Moreover, we can assume that $p_2,p_3 \not \in \mathcal{R}_* \cup \mathcal{A}_*$ are rational points and 
that there exist rational $(n-2)$-dimensional subspaces $L_{p_2}$ and $L_{p_3}$
containing $p_2$ and $p_3$ such that $L_{p_2} \cap \mathcal{A}=\emptyset$ and 
$L_{p_3} \cap \mathcal{A}=\emptyset$ (just choose $L_2$ and $L_3$ to be very close to $L_w$). Finally by slightly deforming $L_{p_2}$ and $L_{p_3}$ we can guarantee that $L_{p_3} \ne hL_{p_2}$, $p_2 \not \in L_{p_3}$
and $p_3 \not \in L_{p_2}$.   Claim \ref{com claim 2} applied to $\mathcal{A}_*$ and $\mathcal{R}_*$
implies that there exists $0<\delta_2<\min(\dist(p_3,L_2),\dist(p_2,L_3))$ for which items \ref{tmp 140} and \ref{tmp 141}  hold.  An additional use of Claim \ref{com claim 2} 
with respect to $\mathcal{A}_*\cup [p_2]_{\delta_2}$ and $\mathcal{R}_*\cup [L_2]_{\delta_2}$  implies that there exists a positive number $\delta_3$ 
such that items \ref{tmp 140} and \ref{tmp 141} hold with respect to 
 $\mathcal{A}_+:=\mathcal{A}_*\cup [p_2]_{\delta_2}\cup [p_3]_{\delta_3}$ and $\mathcal{R}_+:=\mathcal{R}_*\cup [L_2]_{\delta_2}\cup [L_3]_{\delta_3}$. Lemma \ref{throwing} implies that there exists $\mathcal{S}_* \subseteq \mathcal{S}_+$
such that  $(\mathcal{S}_+,\mathcal{A}_+,\mathcal{R}_+)$ is a Schottky system and 
$\langle \mathcal{S}_+,g \rangle=\SL(n,\Z)$.
\end{proof}

We are ready to complete the proof of Proposition \ref{prop dense}.
Let $(g_i)_{i \ge 1}$ be an enumeration of the elements $g$ of $\SL(n,\Z)$ for which there exists an $(n-2)$ dimensional subspace $L $ such that $L \subseteq [L_2]_\rho$ and $gL \subseteq [L_1]_\rho$. Proposition \ref{prop -pro-dense} and Claim \ref{com claim 1}  imply that there exists a profinitely-dense Schottky system $(\mathcal{S}_0,\mathcal{A}_0,\mathcal{R}_0)$
such that $\mathcal{A}_0:=[p_0]_\rho$ and $\mathcal{R}_0:=[L_0 \cup L_1]_\rho$. Claim \ref{com claim 1} allows us to recursively construct an ascending sequence 
$(\mathcal{S}_i,\mathcal{A}_i,\mathcal{R}_i)_{i \ge 0}$ of profinitely-dense Schottky systems such that $\langle \mathcal{S}_i,g_i\rangle=\SL(n,\Z)$ for every 
$i \ge 1$. The set $\cup_{k \ge 0}\mathcal{S}_k$ freely generates a free profinitely-dense subgroup $H$ with the required property. 
%%%%%%%%%%%%%%%%%%%%%%%%%%%%%%%%%%%%%%%%%%%%%%%%%%%%%%%%%%%%%
\section{Proof of Theorems \ref{thm trivial} and \ref{thm not 2-trans}}\label{proof 3}
%%%%%%%%%%%%%%%%%%%%%%%%%%%%%%%%%%%%%%%%%%%%%%%%%%%%%%%%%%%%%

It is enough to show that there exists elements $g,k \in \SL(n,\Z)$ and an infinite-index profinitely-dense subgroup 
$H$ such that $kHk^{-1}\cap H \ne \{\id\}$, $\langle H,k\rangle =\SL(n,\Z)$ and $g Mg^{-1} \cap M =\{\id \}$ for every $H \le M \lneqq \SL(n,\Z)$.
Indeed, if $M$ is any maximal subgroup which contains $H$ then $\SL(n,\Z)$ acts primitively on $\SL(n,\Z)/M$ but the action is not 2-transitive
since the stabiliser of the pair $(M,kM)$ is not trivial while the stabiliser of the pair $(M,gM)$ is trivial.

\begin{defin} We say that the quadruple $(\mathcal{S},\mathcal{A},\mathcal{N},\mathcal{R})$  is a (profinitely-dense) Schottky system if the triple 
$(\mathcal{S},\mathcal{A},\mathcal{R})$  is a (profinitely-dense) Schottky system and $\mathcal{N}$ is an open set which contains $\mathcal{A}$.
\end{defin} 

Before starting with the formal proof let us explain the main idea of the proof. Choose $g \in \SL(n,\Z)$ which fixes some point $p \in \P$. Let  $(\mathcal{S}_0,\mathcal{A}_0,
\mathcal{N}_0,\mathcal{R}_0)$
be a profintiely dense Schottky system such that $g\mathcal{N}_0 \cap \mathcal{N}_0=\emptyset$. Given a non-identity $h \in \SL(n,\Z)$ we would like to construct 
a  Schottky system $(\mathcal{S}_1,\mathcal{A}_1,\mathcal{N}_1,\mathcal{R}_1)$ which contains $(\mathcal{S}_0,\mathcal{A}_0,\mathcal{N}_0,\mathcal{R}_0)$ such that either $\langle \mathcal{S}_0, h \rangle= \SL(n,\Z)$
or $\langle \mathcal{S}_0, g^{-1}hg \rangle= \SL(n,\Z)$. In particular, every maximal subgroup of $\SL(n,\Z)$ which contains 
$\mathcal{S}_1$ does not contain  $\{h,g^{-1}hg\}$. Since $g\mathcal{N}_0 \cap \mathcal{N}_0=\emptyset$, either $h(p) \not \in \mathcal{N}_0$ or  $g^{-1}hg(p) \not \in \mathcal{N}_0$. 
Assume for example that the first case hold. Then we can use the same technique as in the proof of Theorem \ref{thm dense} 
to construct $(\mathcal{S}_1,\mathcal{A}_1,\mathcal{N}_1,\mathcal{R}_1)$ such that $\langle \mathcal{S}_0, h \rangle= \SL(n,\Z)$. In particular, we need to add an element 
$u\in \mathcal{U}$  to $\mathcal{S}_0$ such that $h(p) \in L_u$ and $L_u \cap \mathcal{A}_0=\emptyset$.  
This is the place where we need the open set $\mathcal{N}_0$; $\mathcal{A}_0$ might contain some closed ball $B$ for which 
there exists a sequence of points $(z_n)_{n \ge 1}$ such that $z:=\lim_{n\rightarrow \infty} z_n \in B$ and $z_n \not \in \mathcal{A}_0$ for every $n \ge 1$. If $(L_n)_{n \ge 1}$ is any sequence of hyperplanes such that 
$z_n \in L_n$ and $L_n \cap B =\emptyset$ for every $n \ge 1$ then $\lim_{n\rightarrow \infty} L_n$ is the tangent hyperplane $L$ to $B$ at the point $z$.   The problem is that $L$ might intersect the interior of $\mathcal{A}_0$. Therefore, 
if $h(p)$ is very close to $B$ it might happen that $L \cap \mathcal{A}_0 \ne \emptyset$ for every hyperplane $L$ which contains $h(p)$. In particular, it is not possible to find an element $u$ with the required properties. In order to avoid this problem we use the open set $\mathcal{N}_0$ to guarantee that $h(p)$ is not too close to $\mathcal{A}_0$. Note that we don't need the open sets  in the proof of Theorem \ref{thm dense} because we have a better control on the projective actions of the elements we deal with  (they send some line in a given open set to some other given open set).

ֿ\begin{claim}\label{cal 1} Let $p \in \P$ be fixed by a non-identity element $g \in \SL(n,\Z)$. Let $k \in \SL(n,\Z)$ be an element such that $\id,k,k^2,g,gk,gk^2$
are pairwise distinct elements. Then there exists a profinitely-dense Schottky system 
$(\mathcal{S} , \mathcal{A}, \mathcal{N},\mathcal{R}) $ which satisfies:
\begin{enumerate}
\item $p \not \in \overline{\mathcal{N}} \cup \mathcal{R} \cup g\overline{\mathcal{N}} \cup g\mathcal{R}$ and $\overline{\mathcal{N}} \cap g\overline{\mathcal{N}} =\emptyset$;
\item If $x \not \in \mathcal{N}$ then there exists $x \in L_x \in \mathbb{L}_{n-2}$ and $p_x \in L_x$ such that 
$p_x \not \in \overline{\mathcal{N}}\cup \mathcal{R}$,  $gp_x \not \in \overline{\mathcal{N}}$ and $L_x \cap \mathcal{A}=\emptyset$;
\item $k \langle S \rangle k^{-1} \cap \langle S \rangle\ne  \{\id\}$ and  $\langle S ,k \rangle  =\SL(n,\Z)$.
\end{enumerate}
\end{claim} 
\begin{proof}
Choose rational points $p_1$, $p_2$ and $p_3$
and rational $(n-2)$-dimensional subspaces $L_1$, $L_2$ and $L_3$ such that:
\begin{itemize}
\item For every $1 \le i \ne j \le 3$, $p_i \in L_i$ and $p_i \not \in L_j$;
\item $p_2=kp_1$, $p_3:=k^2p_1$, $L_2=kL_1$, $L_3 \ne k L_1$ and $p_1,p_2,p_3,p_4,p_5,p_6$
are pairwise distinct where $p_4:=gp_1$, $p_5:=gp_2$ and $p_6:=gp_3$;
\item $p \not \in \cup_{1 \le i \le 3}L_i$ (so $p \not \in \cup_{1 \le i \le 3}L_i 
\cup_{1 \le i \le 3}gL_i$ since $gp=p$).
\end{itemize}

A compactness argument implies that there exist $\delta>0$ such that:

\begin{enumerate}
\item[a)] For every $1 \le i \ne j \le 3 $, $[p_i]_\delta \cap [L_j]_\delta=\emptyset$;
\item[b)] $\overline{\mathcal{N}}\cap g\overline{\mathcal{N}}=\emptyset $ where $\mathcal{N}:=\cup_{1 \le i \le 3} (p_i)_{\delta}$;
\item [c)] $p \not \in \overline{\mathcal{N}} \cup \mathcal{R}\cup g\overline{\mathcal{N}} \cup g\mathcal{R}$ where $\mathcal{R}:=\cup_{1 \le i \le 3} [L_i]_{\delta}$;
\item[d)] For every $x \in \P$ there exists $x \in L_x \in \mathbb{L}_{n-2}$ and $p_x \in L_x$ such that 
$p_x \not \in \overline{\mathcal{N}}\cup \mathcal{R}$ and $gp_x \not \in \overline{\mathcal{N}}$.
\end{enumerate}
An additional compactness  argument implies that there exists $\delta>\e>0$ such that: 
\begin{enumerate}
\item[e)] For every $x \not \in \mathcal{N}$ there exists $x \in L_x \in \mathbb{L}_{n-2}$ and $p_x \in L_x$ such that 
$p_x \not \in \overline{\mathcal{N}}\cup \mathcal{R}$,  $gp_x \not \in \overline{\mathcal{N}}$   and $L_x \cap \mathcal{A}=\emptyset$ where
$\mathcal{A}:= \cup_{1 \le i \le 3} [p_i]_\e $.
\end{enumerate}

Items (a), (b), (c), (e) and Corollary \ref{starting} imply that the required  $(\mathcal{S} , \mathcal{A}, \mathcal{N},\mathcal{R})$ exists.
\end{proof}

\begin{claim}\label{cal 2} Let $p \in \P$ be fixed by a non-identity element $g \in \SL(n,\Z)$. Assume that $(\mathcal{S} , \mathcal{A}, \mathcal{N},\mathcal{R}) $ is a profinitely-dense Schottky system which satisfies:
\begin{enumerate}
\item $p \not \in \overline{\mathcal{N}} \cup \mathcal{R} \cup g\overline{\mathcal{N}} \cup g\mathcal{R}$ and $\overline{\mathcal{N}} \cap g\overline{\mathcal{N}} =\emptyset$;
\item\label{item b} If $x \not \in \mathcal{N}$ then there exists $x \in L_x \in \mathbb{L}_{n-2}$ and $p_x \in L_x$ such that 
$p_x \not \in \overline{\mathcal{N}}\cup \mathcal{R}$,  $gp_x \not \in \overline{\mathcal{N}}$ and $L_x \cap \mathcal{A}=\emptyset$;
\end{enumerate}
Then for every non-identity $h \in \SL(n,\Z)$, $(\mathcal{S} , \mathcal{A}, \mathcal{N},\mathcal{R}) $ is contained in 
a Schottky system $(\mathcal{S}_+ , \mathcal{A}_+, \mathcal{N}_+,\mathcal{R}_+) $ which still satisfies assumptions 1 and 2 
and such that $\langle \mathcal{S}_+,h\rangle=\SL(n,\Z)$ or $\langle \mathcal{S}_+,g^{-1}hg\rangle=\SL(n,\Z)$.
\end{claim}
\begin{proof} By replacing $h$ with $g^{-1}hg$ if necessary  we can assume that $y:=hp \not \in \mathcal{N}$. Under this assumption we will show that there exists a Schottky system $(\mathcal{S}_+ , \mathcal{A}_+, \mathcal{N}_+,\mathcal{R}_+) $ which satisfies assumption 1 and 2 and  $\langle \mathcal{S}_+,h\rangle=\SL(n,\Z)$.
Assumption  \ref{item b} implies that there exist  $y \in L_y \in \mathbb{L}_{n-2}$ and $p_y \in L_y$ such that 
$p_y \not \in \overline{\mathcal{N}}\cup \mathcal{R}$,  $gp_y \not \in \overline{\mathcal{N}}$ and $L_y \cap \mathcal{A}=\emptyset$.
By making small deformations we may assume that: $y$ and $L_y$ are rational, $gp_y \ne p_y$  and $p \not \in L_y$. A compactness argument implies that there exists $\delta>0$ such that:
\begin{enumerate}
\item[a)] $[p_y]_\delta\cap \mathcal{R}=\emptyset$ and $[L_y]_\delta\cap \mathcal{A} =\emptyset$;
\item[b)] $\overline{\mathcal{N}_0}\cap g\overline{\mathcal{N}_0}=\emptyset $ where $\mathcal{N}_0:=\mathcal{N} \cup (p_y)_{\delta}$;
\item [c)] $p \not \in \overline{\mathcal{N}_0} \cup \mathcal{R}_0 \cup g\overline{\mathcal{N}_0 }\cup g\mathcal{R}_0$ where $\mathcal{R}_0:=\mathcal{R}_0 \cup [L_y]_\delta$;
\item[d)] For every $x  \not \in \mathcal{N}$ there exists $x \in L_x \in \mathbb{L}_{n-2}$ and $p_x \in L_x$ such that 
$p_x \not \in \overline{\mathcal{N}_0}\cup \mathcal{R}_0$, $gp_x \not \in \overline{\mathcal{N}_0}$ and $L_x \cap \mathcal{A}=\emptyset$.
\end{enumerate}

An additional compactness argument  implies that there exists $\delta>\e>0$ such that: 

\begin{enumerate}
\item[e)] If $x \not \in \mathcal{N}_0$ then there exists $x \in L_x \in \mathbb{L}_{n-2}$ and $p_x \in L_x$ such that 
$p_x \not \in \overline{\mathcal{N}_0}\cup \mathcal{R}_0$,  $gp_x \not \in \overline{\mathcal{N}_0}$   and $L_x \cap \mathcal{A}_0=\emptyset$ where
$\mathcal{A}_0 := \mathcal{A} \cup [p_y]_\e $.
\end{enumerate}

Corollary \ref{adding} and items (a), (b), (c) and (e) implies that there exists $u\in \mathcal{U}$ such that 
$p_u=p_y$, $L_u=L_y$ and $(\mathcal{S}_0 , \mathcal{A}_0, \mathcal{N}_0,\mathcal{R}_0) $ is a Schottky system which still satisfies Assumptions 1 and 2
where $\mathcal{S}_0=\mathcal{S} \cup\{u\}$. Note that $h^{-1}u h p =p$ and that if $(\mathcal{S}_+ , \mathcal{A}_+, \mathcal{N}_+,\mathcal{R}_+) $ contains $(\mathcal{S}_0 , \mathcal{A}_0, \mathcal{N}_0,\mathcal{R}_0) $ and  $\langle \mathcal{S}_+,h^{-1}uh\rangle=\SL(n,\Z)$  then also $\langle \mathcal{S}_+,h\rangle=\SL(n,\Z)$. 

Since $h^{-1}uh$ has infinite order there exists some $m \ge 1$ such that $\id,f,g,gf$ are pairwise distinct elements where $f:=h^{-1}u^mh$. 
Thus, every open neighbourhood of $p$ contains a rational point $p_1$ such that $p_1$, $p_2:=fp_1$, $p_3:=gp_1$ and $p_4:=gfp_1$ are 4 distinct points which are
contained in this neighbourhood. If this neighbourhood does not intersect $\mathcal{N}_0$ then assumption \ref{item b}  implies that for every $1 \le i \le 2$ there exists an $(n-2)$-subspace 
$p_i \in L_i \in \mathbb{L}_{n-2}$  such that $L_i \cap \mathcal{A}_0=\emptyset$. By using small deformations we can assume that the $L_i$'s and $p_i$'s  are rational, $L_2 \ne f L_1$, $p\not \in L_i$  and  
$p_j \not \in L_i$ for every $1 \le i \ne j \le 4$.

A compactness arguments implies that there exists $\delta>0$ such that:

\begin{enumerate}
\item[a)] For every $1 \le i\ne j  \le 2$, $[p_i]_\delta\cap (\mathcal{R}_0 \cup [L_j]_\delta)=\emptyset$ and $[L_i]_\delta\cap \mathcal{A}_0 =\emptyset$;
\item[b)] $\overline{\mathcal{N}_+}\cap g\overline{\mathcal{N}_+}=\emptyset $ where $\mathcal{N}_+:=\mathcal{N}_0 \cup  (p_1)_{\delta} \cup (p_2)_\delta$;
\item [c)] $p \not \in \overline{\mathcal{N}_+} \cup \mathcal{R}_+\cup g\overline{\mathcal{N}_+ }\cup g\mathcal{R}_+$ where $\mathcal{R}_+:=\mathcal{R}_0 \cup  [L_1]_\delta\cup [L_2]_\delta$;
\item[d)] If $x \not \in \mathcal{N}_0$ then there exists $x \in L_x \in \mathbb{L}_{n-2}$ and $p_x \in L_x$ such that 
$p_x \not \in \overline{\mathcal{N}_+}\cup \mathcal{R}_+$, $gp_x \not \in \overline{\mathcal{N}_+}$ and $L_x \cap \mathcal{A}_0=\emptyset$.
\end{enumerate}
An additional compactness argument implies that there exists $\delta>\e>0$ such that:
\begin{enumerate}
\item[e)] If $x \not \in \mathcal{N}_+$ then there exists $x \in L_x \in \mathbb{L}_{n-2}$ and $p_x \in L_x$ such that 
$p_x \not \in \overline{\mathcal{N}_+}\cup \mathcal{R}_0$,  $gp_x \not \in \overline{\mathcal{N}_+}$   and $L_x \cap \mathcal{A}_+=\emptyset$ where
$\mathcal{A}_+ := \mathcal{A} \cup  [p_1]_\e\cup [p_2]_\e $.
\end{enumerate}
Corollary \ref{throwing} and items (a), (b), (c) and (e) imply that there exists $\mathcal{S}_+ \supseteq \mathcal{S}$ such that 
$(\mathcal{S}_+, \mathcal{A}_+,\mathcal{N}_+,\mathcal{R}_+)$ is a Schottky system which satsfies items 1 and 2 and $\langle \mathcal{S}_+ , f \rangle =\SL(n,\Z)$. 
\end{proof}

We are now ready to complete the proof of Theorems \ref{thm trivial} and \ref{thm not 2-trans}. Let $(h_i)_{i \ge 1}$
be an enumeration of the non-identity elements of $\SL(n,\Z)$. Let $p \in \P$ be fixed by a non-identity element $g \in \SL(n,\Z)$. Let $k \in \SL(n,\Z)$ be an element such that $\id,k,k^2,g,gk,gk^2$ are pairwise  distinct
elements. Claims \ref{cal 1} and \ref{cal 2} allow us the recursively construct 
ascending sequence $(\mathcal{S}_i,\mathcal{A}_i,\mathcal{N}_i,\mathcal{R}_i)_{i \ge 0}$ of profintely-dense Schottky systems such that 
$k \langle \mathcal{S}_0 \rangle k^{-1} \cap \langle \mathcal{S}_0 \rangle\ne  \{\id\}$ and  $\langle \mathcal{S}_0 ,k \rangle  =\SL(n,\Z)$ and  for every $i \ge 1$ either $\langle h_i, \mathcal{S}_i\rangle =\SL(n,\Z) $
or $\langle g^{-1}h_ig, \mathcal{S}_i\rangle =\SL(n,\Z) $. The subgroup $H:=\langle \mathcal{S}_i \mid i \ge 0\rangle$ has the required properties stated in the first paragraph of this section. 

\section{The  $\SL(n,\Q)$ case}\label{SL(n,Q)}

Margulis' and Soifer's paper is concerned only with finitely   generated group. In \cite{GG08} the first author and Glasner studied 
infinite index maximal subgroups in general linear groups and proved a criterion for their existence. The goal of this section is
to show that Theorems   \ref{thm trivial}, \ref {thm not 2-trans} and \ref {thm dense} hold also for $\SL(n,\Q)$. Since the proofs are very similar to the ones given for $\SL(n,\Z)$, we only sketch the required modifications needed for proving the $\SL(n,\Q)$ version of Theorem   \ref{thm trivial}. 

The first modification concerns the elements in the generating set $\mathcal{S}$. It is not enough anymore to deal only with unipotent elements. We say that an element $g \in \SL(n,\Q)$ is strongly-semisimple if 
it has a unique eigenvalue $\lambda^+_g$ of maximal absolute value and a unique eigenvalue $\lambda^-_g$ of minimal absolute value. Let $p_g^+$ be the projective point corresponding to the  $\lambda^+_g$-eigenspace and $L^+_g \in\mathbb{L}_{n-1}$ the projective hyperplane corresponding to $\textrm{Im}(\mathrm{Id}-\lambda_g^+)$. Denote 
$p^{-}_g:=p^+_{g^{-1}}$ and $L^{-}_g:=L^+_{g^{-1}}$. To simplify the notation, if $u\in \mathcal{U}$ (in particular $u$ is not strongly-semisimple) we denote $p^+_u=p^-_u:=p_u$ and $L^+_u=L^-_u:=L_u$.

\begin{defin} Assume that $\mathcal{S}$ is a non-empty subset of $\SL(n,\Q)$ consisting of elements which are either strongly-semisimple or unipotent of rank 1.
Let $\mathcal{A} \subseteq \mathcal{R}$ be closed subsets of $\P$. We  call the triple 
$(\mathcal{S},\mathcal{A},\mathcal{R})$ a Schottky system if for every $g \in \mathcal{S}$ there exist two positive
numbers  $\delta_g \ge \e_g$ such that the following properties hold:
\begin{enumerate}
\item $g^{k}(x) \in (p^+_g,p^-_g)_{\varepsilon_g}$ for every $x \in \P\setminus (L^+_g \cup L^-_g)_{\delta_g}$ and every  $k \ne 0$;
\item If $g \ne h \in \mathcal{S}$ then $(p_g^+,p_g^-)_{\varepsilon_g} \cap (L_h^+\cup L_h^-)_{\delta_h}=\emptyset$;
\item $\cup_{g\in \mathcal{S}}(p^+_g,p^-_g)_{\varepsilon_g} \subseteq \mathcal{A}$;
\item $\cup_{g \in \mathcal{S}}(L_h^+, L_h^-)_{\delta_v} \subseteq \mathcal{R}$.
\end{enumerate}
\end{defin} 

The basic idea in the proof of Theorem \ref{thm dense} is to pick two small enough disjoint open subset $U, V \subseteq \P$
and to construct enumeration $(g_k)_{k \ge 1}$ of all the elements $g \in \SL(n,\Z)$ such that $gU\cap V =\emptyset$. (In the proof we passed to the dual space but it is merely a convenience.)  The second step is to construct a unipotent Schottky system  
$(\mathcal{S}_0, \mathcal{A}_0,\mathcal{R}_0)$ such that the following items hold for $k=0$:
\begin{enumerate}
\item[(1)] $\langle \mathcal{S}_k\rangle$ is profintiely dense in $\SL(n,\Z)$;
\item[(2)] $\langle \mathcal{S}_k\rangle \cap \{g_i \mid i \ge 1\}=\emptyset$;
\item[(3)] There exists $\mathcal{S}_{k+1}\supseteq \mathcal{S}_k$, $\mathcal{A}_{k+1}\supseteq \mathcal{A}_k$ 
and $\mathcal{R}_{k+1}\supseteq \mathcal{R}_k$ such that $\langle \mathcal{S}_{k+1},g_{k+1}\rangle =\SL(n,\Z)$ and the first two items still hold for $(\mathcal{S}_{k+1}, \mathcal{A}_{k+1},  \mathcal{R}_{k+1})$.
\end{enumerate}

The main part of the proof of Theorem \ref{thm dense} is to inductively construct an  ascending sequence $(\mathcal{S}_k, \mathcal{A}_k,\mathcal{R}_k)_{k \ge 0}$ such that for every $k\ge 0$ items (1)--(3) hold.  Any infinite index maximal subgroup $\Delta$ of $\SL(n,\Z)$ which contains $\Lambda:=\langle\cup_{k \ge 1}\mathcal{S}_k\rangle$ has empty intersection with $ \{g_k \mid k \ge 1\}$.  The existence of $\Delta$ follows from Zorn's lemma. 

The fact that the $g_k$'s belong to $\SL(n,\Z)$ is not used in the construction, and it is possible to take enumeration 
$(g_k)_{k \ge 1}$ of all the elements $g \in \SL(n,\Q)$ such that $gU\cap V =\emptyset$  and to inductively construct an  ascending sequence $(\mathcal{S}_k, \mathcal{A}_k,\mathcal{R}_k)_{k\ge 0}$ such that for every $k\ge 0$ items (1)--(3) hold where the equality $\langle \mathcal{S}_{k+1},g_{k+1}\rangle =\SL(n,\Z)$ in item (3) is replaced with  $\langle \mathcal{S}_{k+1},g_{k+1}\rangle \supseteq\SL(n,\Z)$.  However, it is not possible to directly apply Zorn's lemma since 
$\SL(n,\Q)$ is not finitely generated and the union of ascending sequence of proper subgroups of $\SL(n,\Q)$ might be equal to $\SL(n,\Q)$. Thus, $\Lambda$ might not be contained in a maximal subgroup. An additional problem is that even if $\Lambda$ is contained in a maximal subgroup then this maximal subgroup might contain $\SL(n,\Z)$ and have a non-empty intersection with $ \{g_k \mid k \ge 1\}$.

In order to overcome these problems we take enumeration $(\mathcal{C}_k)_{k\ge 1}$ of the left cosets of $\SL(\Z[1/p])$ in 
$\SL(n,\Q)$ where $p$ is some odd prime. One start with  a unipotent Schottky system $(\mathcal{S}_0, \mathcal{A}_0,\mathcal{R}_0)$ in $\SL(\Z[1/p])$ such that the following items hold for $k=0$:

\begin{enumerate}
\item[($1^*$)] $\langle \mathcal{S}_k \cap \SL(n,\Z)\rangle$ is profintiely dense in $\SL(n,\Z)$ and $\mathcal{S}_k\setminus \SL(n,\Z) \ne \emptyset$;
\item[$(2^*)$] $\langle \mathcal{S}_k\rangle \cap \{g_i \mid i \ge 1\}=\emptyset$;
\item[$(3^*)$] There exists $\mathcal{S}_{k+1}\supseteq \mathcal{S}_k$, $\mathcal{A}_{k+1}\supseteq \mathcal{A}_k$ 
and $\mathcal{R}_{k+1}\supseteq \mathcal{R}_k$ such that:
\begin{enumerate}
\item  The first two items still hold for $(\mathcal{S}_{k+1}, \mathcal{A}_{k+1},  \mathcal{R}_{k+1})$;
\item $\langle \mathcal{S}_{k+1},g_{k+1}\rangle \supseteq \SL(n,\Z)$;
\item If $k \ne 0$, $\langle \mathcal{S}_{k+1}\rangle\cap \mathcal{C}_k \ne \emptyset$.
 \end{enumerate}
\end{enumerate}

It is possible to inductively construct an  ascending sequence $(\mathcal{S}_k, \mathcal{A}_k,\mathcal{R}_k)$ such that for every $k$ items $(1^*)$--$(3^*)$ hold. The new requirement, Part (c) of item $(3^*)$, brings almost no difficulty since $\SL(n,\Z[1/p])$ is dense in $\SL(n,\R)$ so in every coset it is easy to find an element which still plays ping-pong with the elements of $\mathcal{S}_k$. Note that this is the place where we need to deal with strongly semi-simple elements since an element which is very close to a strongly semi-simple element is strongly semi-simple while an element which is very close to a unipotent element might not be unipotent. Denote $\Lambda:=\langle \mathcal{S}_k \rangle$. 
An ascending sequence of proper subgroups which contain  $\Lambda$ cannot be equal to $\SL(n,\Q)$. Indeed, if this happens then one of the subgroups in the union contains the finitely generated group  $\SL(n,\Z[1/p])$ and has non-trivial intersections with all the cosets of $\SL(n,\Z[1/p])$ so it is equal to
$\SL(n,\Q)$, a contradiction. Zorn's lemma implies that $\Lambda$ is contained in some maximal subgroup $\Delta$. We need to show that  $\Delta \cap \{g_k \mid k \ge 1\} = \emptyset$. Assume otherwise, then item $(3^*)$ implies that 
$\Delta$ contains $\SL(n,\Z)$. Since $\SL(n,\Z)$ is a maximal subgroup of $\SL(n,\Z[1/p])$, item $(1^*)$ implies that $\Delta$
contains $\SL(n,\Z[1/p])$. But $\Delta$ non-trivially intersects all the cosets of $\SL(n,\Z[1/p])$, so $\Delta=\SL(n,\Q)$, a contradiction.

\section{Questions}

\begin{que} \normalfont Does $\SL(n,\Z)$ have a finitely generated infinite-index maximal subgroup? This question appears in Margulis and Soifer's paper and it is still open. Note that general finitely generated linear groups which are not virtually-solvable might but don't have to contain such subgroups. For example, non-commutative free groups do not 
have finitely generated infinite index maximal subgroups. On the other hand, as pointed out by Yair Glasner, the finitely generated linear group $\SL(n,\Z[1/p])$, where $p$ is a prime integer, contains the finitely generated subgroup $\SL(n,\Z)$ as a maximal subgroup. 
\end{que}

\begin{que}\label{question general} \normalfont Margulis and Soifer proved that every finitely generated linear group which is not virtually-solvable contains uncountably many maximal subgroups.  The proof of Theorem \ref{thm counting} uses Venkatramana's theorem which in turn is based on the congruence subgroup property.  Thus, our techniques do not give an answer to  the following question:
Does every  finitely generated linear group which is not virtually-solvable contains $2^{\aleph_0}$  maximal subgroups?
\end{que}

\begin{que}  \normalfont Does $\SL(n,\Z)$ contain two infinite-index non-isomorphic maximal subgroups? 
\end{que}
When $n \ge 4$ Margulis and Soifer proved that $\SL(n,\Z)$ contains an infinite-index maximal subgroup which has a subgroup isomorphic to $\Z^2$.  It is also not hard to construct infinite-index maximal subgroups of $\SL(n,\Z)$ which have non-trivial  torsion elements for $n \ge 3$. Thus, a positive answer to any one of the following three questions 
would provide an example of two infinite-index non-isomorphic maximal subgroups: 

\begin{que}
 Does $\SL(n,\Z)$ contain a torsion-free infinite-index maximal subgroup? 
\end{que}
 
\begin{que}
 Does $\SL(n,\Z)$ contain an infinite-index maximal subgroup
which does not have a subgroup isomorphic to $\Z^2$? 
\end{que}

\begin{que}
Does $\SL(n,\Z)$ contains a maximal subgroup which is a free group? 
\end{que}
Note that it is possible that all the  infinite-index maximal subgroups of a given finitely generated linear group are isomorphic. For example, this is the case for a non-commutative free group.

\begin{que}\label{complex question}  \normalfont An element $g \in \SL(3,\Z)$ is called complex if for every $m \ge 1$ the matrix 
$g^m$ has a non-real eigenvalue.  Is it possible that an infinite-index Zariski-dense subgroup of $\SL(3,\Z)$ 
contain a complex element? The thin subgroups that arise from standard ping-pong arguments on projective space do not contain complex elements. More surprisingly, the thin surface subgroups constructed in the work of Long--Reid--Thistlethwaite 
\cite{LRT11} do not contain complex elements. 
A positive answer to the above question would provide the first example (to the best of our knowledge)  of a thin group whose limit set is $\P$ (In fact this group would act minimally on pairs of distinct elements of projective space). 
A negative answers would provide a very strong restriction on the structure of thin groups.  
\end{que}

\begin{que}  \normalfont Very little is known about permutation actions of $\SL(n,\Z)$. In light of theorem \ref{thm not 2-trans} it is natural to ask if $\SL(n,\Z)$ have 2-transitive actions? More generally, does $\SL(n,\Z)$ have highly transitive actions? In an upcoming joint paper with Glasner we prove that all convergence groups (and in particular, linear groups of real-rank 1) have highly transitive permutation actions. Another question of this nature is: Does $\SL(n,\Z)$ have non-equivalent primitive permutation actions?
\end{que}

\begin{que}  \normalfont Call a subset $S \subseteq \SL(n,\Z)$ fat if  $[\SL(n,\Z):\langle S\rangle]=\infty$ but every 
Zariski-dense subgroup of $\SL(n,\Z)$  which contains $S$ is of finite index. Venkataraman's theorem 
(Theorem \ref{Thm - Venkataramana} above) implies that if $u,v$ are unipotent elements, $u$ has rank 1 and 
$\langle u,v\rangle$ is isomorphic to $\Z^2$ then $\{v,u\}$ is a fat subset. Question \ref{complex question} asks if $\{g\}$ is a fat subset whenever $g$ is a complex element. What are the fat subsets of $\SL(n,\Z)$?
\end{que}


\begin{thebibliography}{1}
\bibitem[AGS14]{AGS14} M.~Aka, T.~Gelander and G.A.~ Soifer.{\it Homogeneous number of free generators.} J. Group Theory 17 (2014), no. 4, 525–539.
\bibitem[CG00]{CG00} J.-P.~Conze and Y.~Guivarc'h, {\it Limit sets of groups of linear transformations.} Ergodic theory and harmonic analysis (Mumbai, 1999). Sankhya Ser. A 62 (2000), no. 3, 367--385.
\bibitem[GG08]{GG08} T.~Gelander and Y.~Glasner, {\it Countable primitive groups.} Geom. Funct. Anal. 17 (2008), no. 5, 1479--1523. 
\bibitem[LRT11]{LRT11}  D.D.~Long, A.~Reid and M.~Thistlethwaite, {\it Zariski dense surface subgroups in SL(3,Z).} Geom. Topol. 15 (2011), no. 1, 1--9. 
\bibitem[Lu99]{Lu99} A.~Lubotzky, {\it One for almost all: generation of SL(n,p) by subsets of SL(n,Z).} Algebra, K-theory, groups, and education (New York, 1997), 125--128, Contemp. Math., 243, Amer. Math. Soc., Providence, RI, 1999.
\bibitem[MS81]{MS81}  G.~Margulis and G.~A.~Soifer,  {\it Maximal subgroups of infinite index in finitely generated linear groups.} J. Algebra 69 (1981), no. 1, 1–-23.
\bibitem[No87]{No87} M.V.~Nori, {\it On subgroups of ${\rm GL}_n(F_p)$.} Invent.
Math. 88 (1987), no. 2, 257--275.
\bibitem[We96]{We96} T.~Weigel, {\it On the profinite completion of arithmetic groups of split type.} Lois d'alg\'{e}bres et vari\'{e}t\'{e}s alg\'{e}briques (Colmar, 1991), 79--101, Travaux en Cours, 50, Hermann, Paris, (1996).
\bibitem[We84]{We84} B.~Weisfeiler, {\it Strong approximation for Zariski-dense subgroups
of semisimple algebraic groups.}  Ann. of Math. (2)  120  (1984),
no. 2, 271--315.
\bibitem[Ve87]{Ve87} T.N.~Venkataramana, {\it Zariski dense subgroups of arithmetic groups.} J. Algebra 108 (1987), no. 2, 325--339.
\end{thebibliography}
\end{document}